\documentclass{amsart}

\usepackage[OT4]{fontenc}
\usepackage[utf8]{inputenc} 
\usepackage{amssymb}
\usepackage{enumerate}
\usepackage{esvect}
\usepackage{mathtools}
\usepackage{bm}

\theoremstyle{plain}
\newtheorem{theorem}{Theorem}[section]
\newtheorem{corollary}[theorem]{Corollary}

\newtheorem{obs}[theorem]{Observation}

\theoremstyle{definition}

\theoremstyle{remark}

\numberwithin{equation}{section}

\DeclareMathOperator{\conv}{conv}
\DeclareMathOperator{\card}{card}

\DeclareMathOperator{\Vol}{Vol}
\renewcommand{\d}[1]{\,\mathrm{d}{#1}}

\begin{document}



\title[Hermite-Hadamard for simplices]{A refinement of the left-hand side of Hermite-Hadamard inequality for simplices}

\author[M. Nowicka]{Monika Nowicka}
\address{Institute of Mathematics and Physics, UTP University of Science and Technology, al. prof. Kaliskiego 7, 85-796 Bydgoszcz, Poland}
\email{monika.nowicka@utp.edu.pl}
\author[A. Witkowski]{Alfred Witkowski}
\email{alfred.witkowski@utp.edu.pl}
\date{16.12.2014}

\begin{abstract} 
In this paper, we establish a new refinement of the left-hand side of Hermite-Hadamard inequality for convex functions of several variables defined on simplices.
\end{abstract}

\maketitle

%



\section{Introduction, definitions and notations}

The classical Hermite-Hadamard inequality \cite{Had} states that if a function $f\colon[a,b]\to\mathbb{R}$ is convex, then
\begin{equation*}
f\left(\frac{a+b}{2}\right)\leq \frac{1}{b-a}\int_a^b f(t)\d{t}\leq \frac{f(a)+f(b)}{2}.
\label{eq:HH-dim1}
\end{equation*}
This inequality has been discussed by many mathematicians.
We refer to \cite{DP,NP} and the references therein.
In the last few decades, several generalizations of the Hermite-Hadamard inequality have been established and studied.
One of them (\cite{Bes}) says that if $\Delta\subset\mathbb{R}^n$ is a simplex with barycenter $\mathbf{b}$ and vertices $\mathbf{x}_0,\dots,\mathbf{x}_n$ and $f\colon\Delta\to\mathbb{R}$ is convex, then
\begin{equation}\label{eq:HH-Bess}
	f(\mathbf{b})\leq \frac{1}{\Vol(\Delta)}\int_\Delta f(\mathbf{x})\d{\mathbf{x}}\leq \frac{f(\mathbf{x}_0)+\dots f(\mathbf{x}_n)}{n+1}.
\end{equation} 
 Wąsowicz and Witkowski in \cite{WW} and 
Mitroi and Spiridon in \cite{MS} investigated the relationship between the left and right-hand sides of \eqref{eq:HH-Bess}. \\
Interesting refinement of both inequalities  in \eqref{eq:HH-Bess} was obtained by Ra{\"\i}ssouli and Dragomir in \cite{RD}. In this paper we use their method to obtain another refinement of the left-hand side of Hermite-Hadamard inequality on simplices.

Before we formulate the main theorem of this paper, we first give some definitions and notations.
For a fixed  natural number $n\geq 1$ let $N=\{0,1,\dots,n\}$. 
Suppose $\mathbf{x}_0,\dots, \mathbf{x}_n\in\mathbb{R}^n$ are such that the vectors $\vv{\mathbf{x}_0\mathbf{x}_i},\ i=1,\dots,n$ are linearly independent.
The set $\Delta=\conv\left\{\mathbf{x}_i\colon  i\in N \right\}$ is called a \textit{simplex}.
Such simplex is an $n$-dimensional object and we shall call it sometimes an $n$-simplex if we would like to emphasize  its dimension.
The point 
\begin{equation*}
\mathbf{b}=\frac{1}{n+1}(\mathbf{x}_0+\dots+\mathbf{x}_n)
\end{equation*}
is called the \textit{barycenter} of $\Delta$.
For any  subset  $K$ of $N$ of cardinality $k\leq n$ we define an $(n-k)$-simplex $\Delta^{[K]}$ as follows.
For each $j\in N\setminus K$ let
\begin{equation}
\mathbf{x}^{[K]}_j=\frac{1}{n+1}\sum_{i\in K} \mathbf{x}_i + \frac{n+1-k}{n+1}\mathbf{x}_j
\label{eq:vertices of Delta[K]}
\end{equation} 
and
\begin{equation}
\Delta^{[K]}=\conv\left\{\mathbf{x}^{[K]}_j\colon  j\in N\setminus K \right\}.
\label{eq:definition of Delta[K]}
\end{equation} 

Obviously $\Delta^{[\emptyset]}=\Delta$ and $\Delta^{[K]}=\mathbf{b}$ if $\card N\setminus K=1$.

The integration over a $k$-dimensional simplex will be always with respect to the $k$-dimensional Lebesgue measure denoted by $\d{\mathbf{x}}$ and the $k$-dimensional volume will be denoted by $\Vol$. There will be no ambiguity, as the dimension will be obvious from the context.

By $H(\mathbf{a},\lambda):\mathbb{R}^n\to \mathbb{R}^n$ we denote the homothety with center $\mathbf{a}$ and scale $\lambda$, given by the formula
$$H(\mathbf{a},\lambda)(\mathbf{x})=\mathbf{a}+\lambda(\mathbf{x}-\mathbf{a}).$$

\section{Refinement of the left-hand side}

This is  the main result of our paper.
\begin{theorem}
	If $f\colon\Delta\to\mathbb{R}$ is a convex function and $K\subset L\subsetneq N$, then
	$$\frac{1}{\Vol \Delta^{[L]}}\int_{\Delta^{[L]}}f(\mathbf{x})\d{\mathbf{x}}\leq \frac{1}{\Vol \Delta^{[K]}}\int_{\Delta^{[K]}}f(\mathbf{x})\d{\mathbf{x}}.$$
	\label{theorem:main}
\end{theorem}
Given the remark stated after the formula \eqref{eq:definition of Delta[K]} it is clear that Theorem \ref{theorem:main} refines the LHS of \eqref{eq:HH-Bess}.

Let us begin with two observations, which will make clear the nature of simplices $\Delta^{[K]}$.

First observation follows immediately from \eqref{eq:vertices of Delta[K]}.
\begin{obs}
All simplices $\Delta^{[K]}$ have common barycenter.
\end{obs}

\begin{obs}
If $K\subset L\subsetneq N$ and $\card L=\card K+1$, then $\Delta^{[L]}$ arises from $\Delta^{[K]}$ in the following way:\\
let $l\in L\setminus K$ and let $\Delta^{[K]}_l$ be the face of $\Delta^{[K]}$ opposite to $\mathbf{x}_l^{[K]}$.
Then

\begin{equation*}
\Delta^{[L]}=H\left(\mathbf{x}_l^{[K]},\frac{n-\card K}{n+1-\card K}\right)\left(\Delta^{[K]}_l\right)
\end{equation*}
\label{obs:2}
\end{obs}
\begin{proof}
	Assume, without loss of generality that $K=\{1,\dots,k\}$ and $L=\{0\}\cup K$. 
Let $k<s\leq n$.
By \eqref{eq:vertices of Delta[K]} the vertices of $\Delta^{[L]}$ are
$$x^{[L]}_{s}=\frac{1}{n+1}\sum\limits_{i=0}^k\mathbf{x}_{i}+\frac{n-k}{n+1}\mathbf{x}_s.$$
Then 
\begin{align*}
x^{[L]}_{s}&=\tfrac{1}{n+1}\sum\limits_{i=1}^k\mathbf{x}_{i}+\tfrac{1}{n+1}\mathbf{x}_0+\tfrac{n-k}{n+1}\mathbf{x}_s\\
&=\tfrac{1}{n+1}\sum\limits_{i=1}^k\mathbf{x}_{i}+\tfrac{n+1-k}{n+1}\mathbf{x}_0+\tfrac{n-k}{n+1}\mathbf{x}_s-\tfrac{n-k}{n+1}\mathbf{x}_0\\
&=\tfrac{1}{n+1}\sum\limits_{i=1}^k\mathbf{x}_{i}+\tfrac{n+1-k}{n+1}\mathbf{x}_0+\tfrac{n-k}{n+1-k}\left(
\tfrac{n+1-k}{n+1}\mathbf{x}_s-\tfrac{n+1-k}{n+1}\mathbf{x}_0\right)\\
&=\mathbf{x}_0^{[K]}+\tfrac{n-k}{n+1-k}\left(\mathbf{x}_s^{[K]}-\mathbf{x}_0^{[K]}\right)
=H\left(\mathbf{x}_0^{[K]},\tfrac{n-k}{n+1-k}\right)\left(\mathbf{x}_s^{[K]}\right).\qedhere
\end{align*}
\end{proof}

Let us brief on the approach proposed by 
 Dragomir and Ra{\"\i}ssouli
in \cite{RD}. They constructed the sequence of subsimplices of $\Delta$ as follows:\\
Let $\mathbf{b}$ be the barycenter of $\Delta$. One can divide $\Delta$ into $n+1$ subsimplices
$$D_i=\conv\{\mathbf{x}_0,\dots,\mathbf{x}_{i-1},\mathbf{b},\mathbf{x}_{i+1},\dots,\mathbf{x}_n\},\quad i=0,1,\dots,n.$$
It is important to note that all these simplices have the same volume.\\ 
Denote by $\mathcal{D}_1$ the set of simplices created this way. The set $\mathcal{D}_{p+1}$ is constructed by applying the above procedure to all simplices in $\mathcal{D}_p$. Dragomir and Ra{\"\i}ssouli proved that for a convex function $f:\Delta\to\mathbb{R}$ one has
\begin{align}
f(\mathbf{b})\leq
\frac{1}{\card \mathcal{D}_p}\sum_{\delta\in\mathcal{D}_p} f(\mathbf{b}_\delta)&\leq
\frac{1}{\card \mathcal{D}_{p+1}}\sum_{\delta\in\mathcal{D}_{p+1}} f(\mathbf{b}_\delta)\notag\\
\intertext{and}
\lim_{p\to\infty} \frac{1}{\card \mathcal{D}_p}\sum_{\delta\in\mathcal{D}_p} f(\mathbf{b}_\delta)&= \frac{1}{\Vol \Delta}\int_\Delta f(\mathbf{x})\d{\mathbf{x}},\label{eq:DR2}
\end{align}
where $\mathbf{b}_\delta$ denotes the barycenter of $\delta$.

We shall use the above procedure to prove the main result of this paper.

\begin{proof}[Proof of Theorem~\ref{theorem:main}.]
	Obviously it is enough to prove the result in case $\card K+1=\card L$.  As above we may assume $K=\{1,\dots,k\}$ and $L=\{0\}\cup K$. Let $\Sigma=\Delta^{[K]}_0$ denote the face of $\Delta^{[K]}$ opposite to $\mathbf{x}_0^{[K]}$. For simplicity denote by $H$ the homothety with center $\mathbf{x}_0^{[K]}$ and scale $\frac{n-k}{n+1-k}$. Then, by Observation \ref{obs:2} we see that $\Delta^{[L]}=H(\Sigma)$.

Let us apply the Dragomir-Ra\"\i ssouli process to $\Sigma$.  Thus we obtain a sequence of sets of subsimplices of $\Sigma$ denoted by $\mathcal{D}_p$. \\
Fix $p\geq 1$. For every $\sigma\in \mathcal{D}_p$ let $\Sigma_\sigma=\conv \left(\sigma\cup\left\{x_0^{[K]}\right\}\right)$. Clearly  the simplices $\Sigma_\sigma$ form a partition of $\Delta^{[K]}$ into simplices of the same height thus  $\Vol \Sigma_\sigma=\Vol\Delta^{[K]} / \card\mathcal{D}_p$.\\
Now we apply the left-hand side of the Hermite-Hadamard inequality to all simplices $\Sigma_\sigma$ to obtain

\begin{align}\label{eq:sum in barycenters}
	\frac{1}{\card \mathcal{D}_p}\sum_{\sigma\in \mathcal{D}_p} f(\mathbf{b}_{\Sigma_\sigma})&\leq \sum_{\sigma\in \mathcal{D}_p}\frac{1}{\Vol \Sigma_\sigma}\int\limits_{\Sigma_\sigma} f(\mathbf{x}) \d{\mathbf{x}}	
		=\frac{1}{\Vol \Delta^{[K]} }\int\limits_{\Delta^{[K]}} f(\mathbf{x}) \d{\mathbf{x}}.
\end{align}
Since $\Delta^{[L]}$ is the image of $\Sigma$ by $H$,  the sets $$H(\mathcal{D}_p)=\{H(\sigma): \sigma\in \mathcal{D}_p\}$$ form the Dragomir-Ra\"\i ssouli sequence for $\Delta^{[L]}$. Moreover, \textit{ comme par miracle} \cite{Prev}, the barycenters of $\Sigma_\sigma$ and that of $H(\sigma)$ coincide, i.e.
\begin{equation}
\mathbf{b}_{\Sigma_\sigma}=\mathbf{b}_{H(\sigma)}.
\label{eq:equality of barycenters}
\end{equation}
From \eqref{eq:sum in barycenters} and \eqref{eq:equality of barycenters} we conclude
\begin{align}\label{eq:sum in barycenters L}
	\frac{1}{\card \mathcal{D}_p}\sum_{\sigma\in \mathcal{D}_p} f(\mathbf{b}_{H(\sigma)})&\leq \frac{1}{\Vol \Delta^{[K]} }\int_{\Delta^{[K]}} f(\mathbf{x}) \d{\mathbf{x}}.
\end{align}
and applying \eqref{eq:DR2} 
\begin{align}\label{eq:limit sum in barycenters L}
	&\lim_{p\to\infty} \frac{1}{\card \mathcal{D}_p}\sum_{\sigma\in \mathcal{D}_p} f(\mathbf{b}_{H(\sigma)})	=	\frac{1}{\Vol \Delta^{[L]} }\int_{\Delta^{[L]}} f(\mathbf{x}) \d{\mathbf{x}}.
\end{align}

Now the assertion follows immediately from  \eqref{eq:limit sum in barycenters L} and \eqref{eq:sum in barycenters L}.

\end{proof}

From Theorem \ref{theorem:main} we obtain the corollary.
\begin{corollary}
Let $K_0, K_1,\ldots, K_{n}$ be a sequence of subsets of $N$ such that
$$K_0\subset K_1 \subset \ldots \subset K_{n} \ \text{and} \ \card K_i=i, \  i=0,1, \ldots, n.$$
If $f:\Delta\to\mathbb{R}$ is convex, then
\begin{align*}
f(\mathbf{b})&=\frac{1}{\Vol(\Delta^{[K_{n}]})}\int_{\Delta^{[K_{n}]}}f(\mathbf{x})\d{\mathbf{x}}\leq \frac{1}{\Vol(\Delta^{[K_{n-1}]})}\int_{\Delta^{[K_{n-1}]}}f(\mathbf{x})\d{\mathbf{x}}\\
&\leq \ldots\leq\frac{1}{\Vol(\Delta^{[K_1]})}\int_{\Delta^{[K_1]}}f(\mathbf{x})\d{\mathbf{x}}\leq\frac{1}{\Vol(\Delta^{[K_0]})}\int_{\Delta^{[K_0]}}f(\mathbf{x})\d{\mathbf{x}}\\
&=\frac{1}{\Vol(\Delta)}\int_{\Delta}f(\mathbf{x})\d{\mathbf{x}}
\end{align*}
(note that $\Vol(\Delta^{[K_i]})$ denotes $(n-i)$-dimensional volume and $\int_{\Delta^{[K_{i}]}}\dots\d{\mathbf{x}}$ denotes integration with respect to $(n-i)$-dimensional Lebesgue measure).
\end{corollary}

Applying Theorem \ref{theorem:main} to all possible proper subsets of $N$ of the same cardinality and summing the obtained inequalities, we obtain the following result.
\begin{corollary}
If $f\colon \Delta\to\mathbb{R}$ is a convex function,
then
$$\frac{1}{\Vol \Delta}\int_{\Delta}f(\mathbf{x})\d{\mathbf{x}}\geq
\frac{1}{\binom {n+1} {k}}\sum_{\substack{\displaystyle{K\subsetneq N}\\{\card K=k}}}\frac{1}{\Vol \Delta^{[K]}}\int_{\Delta^{[K]}}f(\mathbf{x})\d{\mathbf{x}}\,.$$
\end{corollary}
By Theorem \ref{theorem:main} we get the following corollary.
\begin{corollary}
Let $f\colon\Delta\to\mathbb{R}$ be a convex function and let $k< l \leq n$.
Then
$$\frac{1}{\binom {n+1} {l}}\sum_{\mathclap{\substack{{L}\\{\card L=l}}}}\frac{1}{\Vol \Delta^{[L]}}\int_{\Delta^{[L]}}f(\mathbf{x})\d{\mathbf{x}}\leq \frac{1}{\binom {n+1} {k}}\sum_{\mathclap{\substack{{K}\\{\card K=k}}}}\frac{1}{\Vol \Delta^{[K]}}\int_{\Delta^{[K]}}f(\mathbf{x})\d{\mathbf{x}}.$$
\end{corollary}

\begin{proof}
Clearly it is sufficient to prove the corollary only in case $l=k+1$.
Fix $K=\{1,\dots,k\}$.
We have $n+1-k$ oversets of $K$ of cardinality $k+1$.
Applying Theorem \ref{theorem:main} to $K$ and all such oversets and summing the obtained inequalities, we deduce
\begin{align*}
\sum_{\mathclap{\substack{{L\supset K}\\{\card L=k+1}}}}\frac{1}{\Vol \Delta^{[L]}}\int_{\Delta^{[L]}}f(\mathbf{x})\d{\mathbf{x}}\leq (n+1-k) \frac{1}{\Vol \Delta^{[K]}}\int_{\Delta^{[K]}}f(\mathbf{x})\d{\mathbf{x}}.
\end{align*}
Summing this over all possible $K$, we obtain
\begin{equation*}
(k+1)\sum_{\mathclap{\substack{{L}\\{\card L=k+1}}}}\frac{1}{\Vol \Delta^{[L]}}\int_{\Delta^{[L]}}f(\mathbf{x})\d{\mathbf{x}}
\leq(n+1-k)\sum_{\mathclap{\substack{{K}\\{\card K=k}}}}\frac{1}{\Vol \Delta^{[K]}}\int_{\Delta^{[K]}}f(\mathbf{x})\d{\mathbf{x}},
\end{equation*}
since every $L$ has $k+1$ subsets of cardinality $k$. We complete the proof by multiplying both sides by $\frac{k!(n-k)!}{(n+1)!}$.
\end{proof}

\bigskip


\section*{Competing interests}
  The authors declare that they have no competing interests.

\section*{Author's contributions}
   AW came up with an idea, MN extended it and performed all necessary calculations. All authors read and approved the final manuscript.
%



\end{document}